\documentclass[11pt]{amsart}
 \usepackage[margin=3.5cm]{geometry}
\numberwithin{equation}{section}
\usepackage{amssymb}
\usepackage{amsmath, amsfonts,amsthm,amssymb,amscd, verbatim,graphicx,color,multirow,booktabs, caption,tikz,tikz-cd, mathdots,bm}
\usepackage{tikz-cd}
 \usepackage{hyperref} 
\usetikzlibrary{positioning}

\newtheorem{theorem}{Theorem}[section]
\newtheorem{lemma}[theorem]{Lemma}

     \newtheorem*{question}{Question}

      \theoremstyle{definition}

     \theoremstyle{remark}
     \newtheorem{remark}[theorem]{Remark}

\newcommand{\Sym}{\mathop{\mathrm{Sym}}}

\newcommand{\GL}{\mathop{\mathrm{GL}}}
\newcommand{\F}{\mathop{\bf{F}}}

\newcommand{\C}{\mathop{\bf{C}}}
\newcommand{\N}{\mathop{\bf{N}}}

 \definecolor{mycolor}{rgb}{0.55,0.0,0.16}
  \definecolor{myred}{rgb}{0.6,0.0,0.16}
  \definecolor{mygreen}{rgb}{0.0,0.6,0.16}
  \definecolor{myviolet}{rgb}{1,0,1}

\hypersetup{
colorlinks=true,
linkcolor=true,
linktocpage=true,
pageanchor=true,
hyperindex=true
}

 \AtBeginDocument{
     \hypersetup{
  linkcolor=myred,
  urlcolor=mycolor,
citecolor=mygreen
}
     }

\makeatletter
\@namedef{subjclassname@2020}{%
  \textup{2020} Mathematics Subject Classification}
\makeatother


\author[Francesca Lisi and Luca Sabatini]{Francesca Lisi and Luca Sabatini}

\address{Francesca Lisi, University of Firenze} 
\email{francesca.lisi@postecert.it}

\address{Luca Sabatini, University of Warwick} 
\email{luca.sabatini@warwick.ac.uk, sabatini.math@gmail.com}

\begin{document} 
\title[Fixing two points in primitive solvable groups]{Fixing two points\\in primitive solvable groups} 

\maketitle 

\begin{abstract} 
Consider a finite primitive solvable group.
We observe that a result of Y. Yang implies that there exist two points whose pointwise stabilizer has derived length at most $9$.
We show that, if the group has odd cardinality, then there exist two points whose pointwise stabilizer is abelian.
\end{abstract}

\vspace{0.5cm}
\section{Introduction} 

Let $\Omega$ be a finite set and let $G$ be a primitive solvable group over $\Omega$.
By a beautiful theorem of Seress \cite{Ser96}, there always exist four points of $\Omega$ such that their pointwise stabilizer is trivial.
When the group has odd cardinality, three points are sufficient to obtain the same result.
These results are sharp, and in the present paper we ask what can be said when only two points are fixed.

Our beginning is in the observation that a theorem of Moret\'o and Wolf \cite{MW04}[Theorem E] implies the following:
there exist $x,y \in \Omega$ such that the pointwise stabilizer $G_{x,y}$ lies in $\F_{_9}(G_x)$,
the ninth Fitting subgroup of $G_x$, and in $\F_{_2}(G_x)$ when $G$ has odd cardinality.
In particular, the Fitting length of the stabilizer of two points can be bounded by $9$, or $2$, respectively.
These statements have the convenience of saying where $G_{x,y}$ is located with respect to $G_x$,
but are far from the actual truth when we are interested in structural properties of the stabilizer of two points.

In his paper \cite{Yan09}, Y. Yang has refined Moret\'o and Wolf's theorem in two directions:
in fact his results imply that the Fitting length of $G_{x,y}$ can be bounded by $7$, and the derived length by $9$.
We will prove that, when $|G|$ is odd, the stabilizer of two points can be chosen to be abelian.
We summarize these results as follows:

\begin{theorem} \label{th:main}
Let $G$ be a primitive solvable group.
Then there exist two points whose pointwise stabilizer has derived length at most $9$.
If $|G|$ is odd, then there exist two points whose pointwise stabilizer is abelian.
\end{theorem} 

The constant $9$ is probably not the best possible, and lowering this number seems to require a different method (see Remark \ref{remSharp}).
Theorem \ref{th:main} will be consequence of results concerning orbits of linear groups.
The size of such orbits have been studied intensively in the last three decades \cite{Esp91,DJ07,Dol08,Yan20},
often in connection with an influential conjecture of Gluck \cite{Glu85,CHMN15}.
The recent paper \cite{MBR19} contains a detailed summary of the known results,
and provides a strong theorem when $\Sym(4)$ is not involved as a section.

With the exception of \cite{Yan09},
we found that structural properties of the centralizers are much less studied.
We have the following:

\begin{theorem} \label{th:mod} 
Let $G$ be a solvable group and let $V$ be a finite completely reducible faithful $G$-module (possibly of mixed characteristic).
Then
\begin{itemize}
    \item[(i)] there exists $v \in V$ such that the derived length of $\C_G(v)$ is at most $9$;
    \item [(ii)] if $|V||G|$ is odd, there exists $v \in V$ such that $\C_G(v)$ is abelian.
\end{itemize}
\end{theorem} 

We prove Theorem \ref{th:mod}(ii) in the next section.
Theorem \ref{th:mod}(i) follows from the stronger result in \cite{Yan09}, but we will give a direct (and easier) proof below.
Since there are irreducible linear groups of any Fitting length, the Fitting length of the stabilizer of one point is not bounded by any constant
(and so is not the derived length).
Similarly, the cardinality of the stabilizer of two points is not bounded,
because there exist primitive solvable subgroups of $\Sym(n)$ of order roughly $24^{-1/3} n^{3.24}$ \cite{Pal82,Wol82},
while the index of such a stabilizer is bounded above by $n(n-1)$.

\vspace{0.1cm}
\section{Proofs} \label{secProofs} 

\subsection{Preliminaries} 

It is well known that a primitive solvable group over $\Omega$ is of the type $V \rtimes G$,
where $|V|=|\Omega|=p^n$ and $G$ is the stabilizer of a point.
Moreover, $V$ is a faithful irreducible $G$-module, and the stabilizers of two points correspond to the centralizers $\C_G(v)$ for some $v \in V$.
In particular, Theorem \ref{th:main} is equivalent to the irreducible case of Theorem \ref{th:mod}.

Let $G$ be a finite group and let $V$ be a completely reducible faithful $G$-module.
Recall that $V$ is quasiprimitive if $V_N$ is homogeneous for all normal subgroups $N$ of $G$.
If $q$ is a prime power and $m \geq 1$, the semilinear group $\mathrm{\Gamma}(q^m)$ is isomorphic to the semidirect product
$\mathrm{\Gamma}_0(q^m) \rtimes \mathrm{Gal}(\mathbb{F}_{q^m}/\mathbb{F}_q)$,
where $\mathrm{\Gamma}_0(q^m) \cong (\mathbb{F}_{q^m})^*$ and $\mathrm{Gal}(\mathbb{F}_{q^m}/\mathbb{F}_q) \cong C_m$ is cyclic.
As the next result shows, semilinear groups play an important role in the study of quasiprimitive groups.

\begin{lemma} \label{lemSemiL}
Let the solvable finite group $G$ act faithfully and quasiprimitively on a vector space $V$.
If $\F(G)$ is abelian, then $V=\mathbb{F}_{q^m}$ for some $q,m \geq 1$, $G \leqslant \mathrm{\Gamma}(q^m)$,
and $\F(G) \leqslant \mathrm{\Gamma}_0(q^m)$.
\end{lemma} 
\begin{proof}
This is a combination of Corollary 2.3 and Theorem 2.10 in \cite{MW93}.
\end{proof} 

If a group $S$ permutes a set $\Omega$, then $S$ also permutes the power set $\mathcal{P}(\Omega)$ of $\Omega$ in a natural way.
The next is \cite[Lemma 4.1]{MW04}.

\begin{lemma}[Gluck's permutation lemma] \label{lemGluck}
   Let $S$ be a transitive solvable permutation group on $\Omega$, $|\Omega|=m$.
   Then
    \begin{itemize}
        \item[(i)] if $|S|$ is odd, then $S$ has a regular orbit on $\mathcal{P}(\Omega)$;
        \item[(ii)] if $S$ is primitive and $m \geq 10$, then $S$ has at least $8$ regular orbits on $\mathcal{P}(\Omega)$
        and at least one regular orbit of subsets $\Delta$ of $\Omega$ such that $\Delta$ is not $S$-conjugate to its complement $\Omega \setminus \Delta$;
        \item[(iii)] if $S$ is primitive, then there exist subsets $\Delta_1,\ldots,\Delta_k$, $k \leq 4$,
        such that $\Omega$ is the disjoint union of the $\Delta_i$'s, and $\cap_i Stab_S(\Delta_i) =1$.
    \end{itemize}
\end{lemma}

We report the following observation, which makes our proof of Theorem \ref{th:mod} particularly simple.
When $\ell \geq 0$, let $G^{(\ell)}$ denote $\ell$-th term of the derived series of $G=G^{(0)}$.

\begin{lemma} \label{lemGT} 
Let $G$ be a group,
and $J_i$ and $K_i$ subgroups such that $J_i^{(\ell)} \subseteq K_i$ for each $i=1,\ldots,n$.
Suppose that there exists a subgroup $A$ such that $J_i \subseteq A$ for each $i$,
and also $(\cap_i K_i) \cap A =1$.
Then $\cap_i J_i$ is solvable with derived length at most $\ell$.
\end{lemma} 
\begin{proof}
We can write
\begin{align*}
(\cap_i J_i)^{(\ell)} & \> = \> 
(\cap_i J_i)^{(\ell)} \cap A \\ & \> \subseteq \> 
(\cap_i J_i^{(\ell)}) \cap A \\ & \> \subseteq \> 
\left( \cap_i \> K_i \right) \cap A \\ & \> = \> 
  1 \> ,
\end{align*} 
from which the proof follows.
\end{proof}

\subsection{The odd case} \label{subsecOdd} 

We prove Theorem \ref{th:mod}(ii).
We follow a fairly standard route: we reduce to $V$ irreducible,
then we handle $V$ quasiprimitive, and the non-quasiprimitive case follows from induction and Gluck's permutation lemma.

Let $V=V_1 \oplus V_2$, and suppose that there exists $v_i \in V_i$ such that $\C_G(v_i)/\C_G(V_i)$ is abelian, for each $i=1,2$.
Let $v=v_1 + v_2 \in V$, so that $\C_G(v)=\C_G(v_1) \cap \C_G(v_2)$.
We can apply Lemma \ref{lemGT} with $n=2$, $\ell=1$, $A=G$, $J_i=\C_G(v_i)$ and $K_i=\C_G(V_i)$ for each $i$.
We conclude that $\C_G(v)$ is abelian, as desired.

We give a neat result when $V$ is quasiprimitive:

\begin{lemma} \label{lemOddQP} 
Let $G$ be a solvable group and let $V$ be a quasiprimitive faithful $G$-module.
If $|V||G|$ is odd, then there exists $v \in V$ such that $\C_G(v)$ is cyclic.
\end{lemma} 
\begin{proof}
Since $|V||G|$ is odd, a useful theorem of Espuelas \cite{Esp91}[Lemma 2.1] states that either $G$ has a regular orbit on $V$, or $F=\F(G)$ is cyclic.
In the first case, $\C_G(v)=1$ for some $v \in V$ and the result follows.
In the latter case, we actually show that $\C_G(v)$ is cyclic for all $v \in V \setminus \{0\}$.
We apply Lemma \ref{lemSemiL}, so that $G \leqslant \Gamma(q^m)$, $F \leqslant \mathrm{\Gamma}_0(q^m)$.
In particular, it is easy to see that $F= \mathrm{\Gamma}_0(q^m) \cap G$, so $G/F \cong \mathrm{\Gamma}_0(q^m) G /\mathrm{\Gamma}_0(q^m)$
can be seen as a subgroup of $\mathrm{\Gamma}(q^m)/\mathrm{\Gamma}_0(q^m) \cong C_m$.
 Let $v \in V \setminus \{0\}$,
 and observe that $\C_F(v)=1$, because $F$ is abelian and regular on $V$.
It follows that
$$ {\bf C}_G(v) \cong \frac{\C_G(v)}{F \cap \C_G(v)} 
\cong \frac{F \C_G(v)}{F} \leqslant G/F , $$
and $\C_G(v)$ is cyclic.
\end{proof} 

Even under the conditions of Lemma \ref{lemOddQP}, there is no $v \in V$ such that $\C_G(v) \subseteq \F(G)$ in general.
(For example, take $G \cong C_{31} \rtimes C_3$ as a subgroup of $\GL_3(5)$,
where the centralizers range among the complements of the Fitting subgroup).
Looking for purely structural properties of the centralizers gives a much better result at this stage.

\begin{proof}[Proof of Theorem \ref{th:mod}(ii)]
Let $V$ be non-quasiprimitive, and proceed by induction on $|V|+|G|$.
Choose $N \lhd G$ maximal such that $V_{N} = V_1 \oplus \cdots \oplus V_m$ is not homogeneous.
By \cite[Proposition 0.2]{MW93}, $S=G/N$
faithfully and primitively permutes the $V_i$'s.
We also observe that $\cap_i \> \C_N(V_i) =1$ because $V$ is faithful.

Lemma \ref{lemGluck}(i)
provides a subset of $\left\{ V_{1},\ldots,V_{m}\right\} $
that lies in a regular orbit of $S$ on the power set of $\left\{ V_{1},\ldots,V_{m}\right\} $.
Without loss of generality, we may assume that
$$ Stab_S( \{ V_1, \ldots, V_k \} )=1 $$
for some $1\leq k<m$. 
Set $H_i=\N_G(V_i)$ for each $i$ (note that $H_i \neq G$ by the irreducibility of $V$).
Also $N=\cap_i H_i$, and the $H_i/N$'s all are conjugates by $S$.
Observe that $\C_G(u) \subseteq H_1$ for all $u \in V_1 \setminus \{0\}$,
and that $V_i$ is an irreducible $H_i$-module.
Applying the inductive hypothesis to the action of $H_1$ on $V_1$,
we obtain $v_1 \in V_1$ such that $\C_G(v_1)/\C_G(V_1)$ is abelian.

 The elements of $V_i$ conjugate to $v_1$ form an $H_i$-orbit on $V_i$.
For each $i>1$, pick $v_i \in V_{i}$ conjugate to $v_1$.
 Let $v = v_1+\ldots+v_k -v_{k+1} - \ldots -v_m$.
 Since $|V||G|$ is odd, no $v_i$ is $G$-conjugate to $-v_j$ (not even for $i=j$)
and so $\C_G(v)$ must stabilize $\{ V_1,\ldots,V_k \}$.
From above, we obtain $\C_G(v) \subseteq N$.

To sum up, we have that $\C_G(v_i)/\C_G(V_i)$ are all abelian sections,
$\C_N(v) = \cap_i \C_N(v_i)$ because $N$ does not permute the $V_i$'s, and also $\C_G(v) \subseteq N$.
In particular,
$$ {\bf C}_N(v_i)^{(1)} \subseteq {\bf C}_G(v_i)^{(1)} \subseteq {\bf C}_G(V_i) . $$
We are in the position to apply Lemma \ref{lemGT} with $\ell=1$, $A=N$, $J_i=\C_N(v_i)$ and $K_i=\C_G(V_i)$ for each $i$.
We obtain that $\cap_i \C_N(v_i)=\C_N(v) = \C_G(v)$ is abelian.
This concludes the proof of Theorem \ref{th:mod}(ii).
\end{proof}

\begin{remark} 
Trivially, $V = (\mathbb{F}_3)^2$ and $G=\GL_2(3)$ show that the hypothesis that $|G|$ is odd is necessary in Theorem \ref{th:mod}(ii).
The hypothesis that $|V|$ is odd is also necessary:
it can be checked with the GAP System \cite{GAP} that \texttt{PrimitiveGroup($2^9$,$28$)} provides a bad example
with $V=(\mathbb{F}_2)^9$ and $|G|=9 \, 261$.
\end{remark}

\subsection{The general case} \label{subsecGen}

In what follows, $G$ is always a solvable group and $V$ is a finite completely reducible faithful $G$-module.
Yang \cite[Theorem 3.4]{Yan09} shows that there exists $v \in V$ such that $\C_G(v)$ is contained in a normal subgroup of $G$ of derived length at most $9$.
As stated, this result is best possible (see \cite[Sec. 4]{Yan09}).
On the other hand, it does not really say more than Theorem \ref{th:mod}(i), in the context of Theorem \ref{th:main}.
We use a different approach that looks purely at the structure of $\C_G(v)$.
We say that an orbit $v^G \subseteq V$ is {\itshape good} if the derived length of $\C_G(v)$ is at most $9$.
For an induction argument, we need more than one good orbit,
so we actually prove Theorem \ref{thEven} below.

\begin{theorem} \label{thEven} 
Let $G$ be a solvable group and let $V$ be a finite completely reducible faithful $G$-module.
    There exist at least two good $G$-orbits.
    If there are less than five good $G$-orbits, then the derived length of $G$ is at most $6$.
\end{theorem}

We remark that in the last case every $G$-orbit is a good orbit, so there are less than five $G$-orbits in total.
It is sufficient to prove Theorem \ref{thEven} when $V$ is irreducible, but a little more care is necessary.
If $V=V_1 \oplus V_2$, we apply Lemma \ref{lemGT} as in the previous section, with $\ell=9$ and $A=G$.
If there are $a \geq 2$ good orbits of $G/\C_G(V_1)$ on $V_1$, and $b \geq 2$ good orbits of $G/\C_G(V_2)$ on $V_2$,
we obtain at least $ab$ good orbits of $G$ on $V$.
Now $ab \geq 5$, except when $(a,b)=(2,2)$.
In this case, by induction, the derived length of $G/\C_G(V_i)$ is at most $6$,
and it is easy to see that $G$ is a subgroup of $G/\C_G(V_1) \times G/\C_G(V_2)$.
So $G$ has derived length at most $6$, 	every $G$-orbit on $V$ is a good orbit, and it is clear that we can assume that $V$ is irreducible.

Let $r$ be the number of $G$-orbits on $V$
(so that $r \geq 2$ when $V$ is nontrivial).
Let $dl(G)$ denote the derived length of $G$.
We report two Yang's lemmas:

\begin{lemma}[Theorem 3.1 in \cite{Yan09}] \label{lemEvenQuasiprimitive}
If $V$ is quasiprimitive, and $dl(G) \geq 10$, then there exist at least five regular $G$-orbits.
\end{lemma}

\begin{lemma}[Lemma 2.5 in \cite{Yan09}] \label{lemR}
If $r \leq 2$, then $dl(G) \leq 4$.
If $r \leq 4$, then $dl(G) \leq 6$.
\end{lemma}

It can be checked with the GAP System \cite{GAP} that there is a subgroup $G$ of $\GL_6(2)$ satisfying $r=3$ and $dl(G)=6$. 
We note that the proofs of these lemmas appear a bit tedious only if one cares about constants:
lossy versions follow immediately from \cite[Corollary 1.10]{MW93} and \cite[Theorem 1.2]{Fou69}, respectively.

\begin{proof}[Proof of Theorem \ref{thEven}]
We work by induction on $|V|+|G|$.
If $dl(G) \leq 9$, then every $G$-orbit is a good orbit, and so we are done by Lemma \ref{lemR}.
Let $dl(G) \geq 10$.
By Lemma \ref{lemEvenQuasiprimitive}, we can also suppose that $V$ is not quasiprimitive.
We will actually show that in the remaining cases there are at least five good $G$-orbits.

There exists $N \lhd G$ such that $V_N = V_1 \oplus \cdots \oplus V_m$ for $m \geq 2$ homogeneous components $V_i$'s of $V_N$;
by choosing $N$ maximal such, then $S=G/N$ primitively permutes the $V_i$'s. 
If $H=\N_G(V_1)/\C_G(V_1)$,
then $H$ acts faithfully and irreducibly on $V_1$ and $G$ is isomorphic to a subgroup of $H \wr_m S$ by the imprimitive embedding theorem.
(Notice that $H$ is faithful by definition, unlike $H_1$ in the previous section).

For a moment, suppose that we have an element $v=v_1+\cdots+v_m \in V$ such that $dl(\C_G(v_i)/\C_G(V_i)) \leq 9$ for each $i$,
$\C_N(v) = \cap_i \C_N(v_i)$, and also that $\C_G(v) \subseteq N$.
Then $\C_N(v_i)^{(9)} \subseteq \C_G(v_i)^{(9)} \subseteq \C_G(V_i)$.
So we are in the position to apply Lemma \ref{lemGT} with $\ell=9$, $A=N$, $J_i=\C_N(v_i)$ and $K_i=\C_G(V_i)$ for each $i$.
We obtain that $\cap_i \C_N(v_i)=\C_N(v) = \C_G(v)$ has derived length at most $9$, i.e. $v$ lies in a good $G$-orbit.

It rests to find five good elements which are not $G$-conjugates. This is related to finding a sufficient number of good $H$-orbits on $V_1$.
If $x_1,y_1 \in V_1$ are in distinct $H$-orbits and $x_i,y_i \in V_i$ are their conjugates,
then the centralizer in $G$ of $v=x_1+\cdots+x_j+y_{j+1}+\cdots+y_m$ must stabilize $\Delta=\{V_1,\ldots,V_j\}$,
if $\Delta$ is not $S$-conjugate to its complement.

In the last part, we do not write all the details and we refer the reader to \cite[Theorem 4.6]{MW04} because the argument is the same.
If $m \geq 10$, then by induction we have $s \geq 2$ good $H$-orbits on $V_1$. We can effectively use them by Lemma \ref{lemGluck}(ii),
and so we get at least $s(s-1)$ good $G$-orbits on $V$.
Now $s(s-1) \geq 5$, except when $s=2$.
Even in this case, we can construct at least five good $G$-orbits on $V$, using Lemma \ref{lemGluck}(iii).
Finally, suppose $m \leq 9$. We have
\begin{equation} \label{eqFinal}
10 \leq dl(G) \leq dl(H) +dl(S) .
\end{equation}
Let $s \geq 2$ be the number of good $H$-orbits on $V_1$.
Using Lemma \ref{lemGluck}(iii) to construct good $G$-orbits, we can easily assume that we are in one of two situations:
\begin{itemize}
\item $s=2$. By Lemma \ref{lemR} we have $dl(H) \leq 4$, while it is easy to see that $dl(S) \leq 5$;
\item $s \in \{3,4\}$ and $m \leq 5$.  By Lemma \ref{lemR} we have $dl(H) \leq 6$, while it is easy to see that $dl(S) \leq 3$;
\end{itemize}
In both cases, we obtain a contradiction to (\ref{eqFinal}).
This concludes the proof of Theorem \ref{thEven} and so of Theorem \ref{th:mod}(i).
\end{proof}

\begin{remark} \label{remSharp} 
As we have said in the Introduction, we believe that the constant $9$ in Theorem \ref{th:main} is not the best possible,
and we are not aware of examples where the derived length of $G_{x,y}$ always exceeds $4$.
The bottleneck in Moret\'o and Wolf's method is because Lemma \ref{lemGluck} is less effective in the even case,
which forces to formulate a statement as Theorem \ref{thEven}.
After this, we observe that Lemma \ref{lemEvenQuasiprimitive} does not really play a role in pinning down the constant,
Lemma \ref{lemR} is sharp, and we come very close to equality in (\ref{eqFinal}).
\end{remark}

 More recently, if $G$ is a primitive group with solvable stabilizer,
 Burness \cite{Bur21} has proved that there always exist five points such that their pointwise stabilizer is trivial.
This result and Theorem \ref{th:main} suggest the following question:
 
 \begin{question}
 Let $G \leqslant \Sym(\Omega)$ be a primitive group with a solvable stabilizer.
 Do there exist $x,y \in \Omega$ such that the derived length of $G_{x,y}$ is bounded by an absolute constant?
 \end{question}

\end{document}